\documentclass[11pt,notitlepage,a4paper]{article}

\usepackage{amssymb} 
\usepackage[intlimits]{amsmath}
\usepackage{amsfonts}
\usepackage{amsthm} 

\usepackage{mathptmx}

\usepackage{hyperref}
\hypersetup{colorlinks=true,linkcolor=RoyalPurple,citecolor=RoyalPurple}

\usepackage{cite}

\usepackage{graphicx}
\usepackage{geometry}
\usepackage{color}
\usepackage[dvipsnames]{xcolor}

\let\oldsqrt\sqrt
\def\sqrt{\mathpalette\DHLhksqrt}
\def\DHLhksqrt#1#2{%
\setbox0=\hbox{$#1\oldsqrt{#2\,}$}\dimen0=\ht0
\advance\dimen0-0.2\ht0
\setbox2=\hbox{\vrule height\ht0 depth -\dimen0}%
{\box0\lower0.4pt\box2}}

\allowdisplaybreaks

\newcommand{\R}{\mathbb{R}} % reelle Zahlen
\newcommand{\N}{\mathbb{N}} % natuerliche Zahlen
\newcommand{\inn}{\textnormal{int}} % interior ...
\newcommand{\dist}{\textnormal{dist}} % dist ...
\renewcommand{\phi}{\varphi}

\newcommand{\cBa}{{{\mathcal B}_{\alpha}}}

\newcommand{\cL}{{\mathcal L}}

\newcommand{\x}{\bar{x}}
\newcommand{\rn}{\R^N}
\newcommand{\calL}{\mathcal{L}_{\alpha}^1}

\newcommand{\eps}{\varepsilon}

\theoremstyle{definition}
\newtheorem{defi}{Definition}[section]
\newtheorem{remark}[defi]{Remark}

\theoremstyle{plain} %default%plain
\newtheorem{thm}[defi]{Theorem}

\newtheorem{lemma}[defi]{Lemma}
\newtheorem{cor}[defi]{Corollary}

\theoremstyle{definition}

\numberwithin{equation}{section} 
\title{Starshape of the superlevel sets of solutions to equations involving the fractional Laplacian in starshaped rings}
\author{
 \ Sven Jarohs\footnote{Goethe-Universit\"at Frankfurt, Germany, jarohs@math.uni-frankfurt.de.},
 \;   
\ Tadeusz Kulczycki\footnote{Wroclaw University of Science and Technology, Poland, Tadeusz.Kulczycki@pwr.edu.pl,\newline T. Kulczycki was supported in part by the National Science Centre, Poland, grant no. 2015/17/B/ST1/01233}, 
\; and
\ Paolo Salani\footnote{DiMaI, Universit\`a di Firenze, Italy, paolo.salani@unifi.it	 }
}

\date{\today}

\begin{document}
\maketitle

\begin{abstract}
In the present work we study solutions of the problem
\begin{equation}
	\begin{aligned}
		-(-\Delta)^{\alpha/2}u&=f(x,u) &&\text{ in $D_0\setminus \overline{D}_1$}\\
		u&=0&& \text{in $\R^N\setminus D_0$,}\\
		u&=1&& \text{in $\overline{D}_1$,}
	\end{aligned}
\end{equation}
where $D_1, D_0\subset \R^N$ are open sets such that $\overline{D}_1\subset D_0$, $\alpha\in(0,2)$, and $f$ is a nonlinearity.
Under different assumptions on $f$ we prove that, if $D_0$ and $D_1$ are starshaped with respect to the same point $\x\in\overline{D}_1$, then the same occurs for every superlevel set of $u$.
\end{abstract}
{\footnotesize
\begin{center}
\end{center}
\begin{center}
%\textit{Mathematics Subject Classification:} %(2011)
% 35K58 %semilinear parabolic equations
%$\cdot$ 35B40 %asymptotic behaviour 
\end{center}
}

\section{Introduction}
In this work we investigate for $\alpha\in(0,2)$ the geometry of solutions $u$ to the problem
\begin{equation}\label{main-prob}
	\begin{aligned}
		-(-\Delta)^{\alpha/2}u&=f(x,u) &&\text{ in $D_0\setminus\overline{D}_1$}\\
		u&=0&& \text{in $\R^N\setminus D_0$,}\\
		u&=1&& \text{in $\overline{D}_1$,}
	\end{aligned}
\end{equation}
where $D_1, D_0\subset \R^N$ are open sets such that $\overline{D}_1\subset D_0$ and $f$ is a bounded Borel function on $(D_0\setminus \overline{D}_1)\times [0,1]$. Moreover, $(-\Delta)^{\alpha/2}$ is the fractional Laplacian, which is defined for $f \in \calL$ and $x \in \rn$ by
\begin{equation}
\label{deffrac}
(-\Delta)^{\alpha/2}u(x) =c_{N,\alpha} \lim_{\eps \downarrow 0} \int_{|x-y|>\epsilon} \frac{u(x) - u(y)}{|x - y|^{N + \alpha}} \, dy,
\end{equation}
whenever the limit exists, where $c_{N,\alpha} = 2^{\alpha-2}\pi^{-\frac{N}{2}}\alpha(2-\alpha)\frac{\Gamma(\frac{N+\alpha}{2})}{\Gamma(2-\frac{\alpha}{2})}$ is a normalization constant and by $\calL$ we denote the space of all Borel functions $u: \rn \to \R \cup \{\infty\}$ satisfying
$$
\int_{\rn} \frac{|u(x)|}{(1+|x|)^{N+\alpha}} \, dx < \infty.
$$
It is well-known that $(-\Delta)^{\alpha/2} \varphi(x)$ is well-defined for any $\varphi \in C_c^2(\rn)$ and $x \in \rn$. For more details and basic properties of the fractional Laplacian, we refer the interested reader to the following recent survey papers \cite{AV17,Garofalo}, which contain also comprehensive bibliographies.\\

For $u \in \calL$ we define the distribution $(-\Delta)^{\alpha/2} u$ by the formula
$$
\langle (-\Delta)^{\alpha/2}u, \varphi \rangle = \langle u, (-\Delta)^{\alpha/2} \varphi \rangle \quad\text{ for $\phi\in C^{\infty}_c(\rn)$,}
$$
(cf. Definition 3.7 in \cite{BB1999}). We say that a function $u$ is a {\it{solution}} of (\ref{main-prob}), if $u$ is continuous and bounded on $\rn$, $u = 0$ in $\rn \setminus D_0$, $u = 1$ in $\overline{D}_1$, and $-(-\Delta)^{\alpha/2}u = f(x,u)$ as distributions in $D_0 \setminus \overline{D}_1$, i.e.
$$
\langle -(-\Delta)^{\alpha/2}u, \varphi \rangle = \langle f(x,u),  \varphi \rangle \quad\text{ for $\phi\in C^{\infty}_c(D_0 \setminus \overline{D}_1).$}
$$
In other words the restriction of $-(-\Delta)^{\alpha/2}u$ to $D_0 \setminus \overline{D_1}$ is a function in $L^{\infty}(D_0 \setminus \overline{D_1})$ and we have 
\begin{equation}
\label{weak1}
-(-\Delta)^{\alpha/2}u(x) = f(x,u(x)) \quad \text{for a.e.} \quad x \in D_0 \setminus \overline{D_1}.
\end{equation}

The geometric properties of solutions to equations involving fractional Laplacians have been recently intensively studied. The results concern concavity properties of the first eigenfunction \cite{BKM2006}, \cite{BB2015}, \cite{KS}, concavity properties of solutions of the Dirichlet problem for $(-\Delta)^{1/2} \varphi = 1$ \cite{K2014}, convexity of superlevel sets for some problems for $(-\Delta)^{1/2}$ \cite{NR2015}, convexity properties of solutions of $(-\Delta)^{\alpha/2} u = f(u)$ \cite{G2015}, and general symmetry properties (see e.g. \cite{BLW05,FQT12,JW13,SV14,J16}) in the spirit of Gidas, Ni and Nirenberg \cite{GNN79}.

Here we are interested in the starshapedness of the level sets of solutions in starshaped rings. Then let us introduce some notation and definitions.

We recall that a subset $A$ of $\rn$ is said {\em starshaped with respect to the point} $\x\in A$ if for every $x\in A$ the segment $(1-s)\x+s x$, $s\in[0,1]$, is contained in $A$. If $\x=0$ (as we can always assume up to a translation), we simply say that $A$ is {\em starshaped}, meaning that for every $x\in A$ we have $sx\in A$ for $s\in[0,1]$, or equivalently
\begin{equation}\label{stardef}
A\text{ is starshaped if }\,\,sA\subseteq A\,\,\text{ for every }s\in[0,1]\,.
\end{equation}
$A$ is said {\em strictly starshaped} if $0$ is in the interior of $A$ and any ray starting from $0$ intersects the boundary of $A$ in only one point.
We say that $A$ is {\em uniformly starshaped} if the exterior unit normal $\nu(x)$ exists at each $x\in\partial A$ and there exists $\epsilon>0$ such that $\langle x,\nu(x)\rangle\geq \epsilon\text{ for every }x\in\partial A$. 
 
By $U(\ell)$, $\ell \in \R$ we denote the superlevel sets of a function $u$:
$$
U(\ell):=\{u\geq \ell\}=\{x\in\rn\,:\, u(x)\geq \ell\}\,.
$$

In the formulation of our results we will use the following conditions on $D_0, D_1\subset \R^N$ and $f: (D_0\setminus \overline{D}_1)\times [0,1] \to \R$:
\begin{enumerate}
	\item[(D)] $D_0, D_1\subset \R^N$ are open sets such that $0\in D_1$, $\overline{D}_1\subset D_0$, and $D_0\setminus \overline{D_1}$ satisfies a uniform exterior cone condition.
		\item[(F0)] $f$ is a bounded Borel function on $(D_0\setminus \overline{D}_1)\times [0,1]$.
		\item[(F1)] $t^{\alpha} f(tx,u) \ge f(x,u)$ for every $t \ge 1$ and $(x,u) \in (D_0\setminus \overline{D}_1)\times [0,1]$ such that $tx \in D_0 \setminus \overline{D}_1$;
		\item[(F2)] $f$ is Lipschitz in the second variable i.e. there exists $C > 0$ such that for any $x \in D_0\setminus \overline{D}_1$, $u_1, u_2 \in [0,1]$ we have $|f(x,u_1) - f(x,u_2)| \le C |u_1 - u_2|$.
		\item[(F3)]  $f$ is increasing in the second variable i.e. $f(x,u_1) \le f(x,u_2)$ whenever $u_1 < u_2$ for any $x \in D_0\setminus \overline{D}_1$, $u_1, u_2 \in [0,1]$.
		\item[(F4)] $f$ is a bounded continuous function on $(D_0\setminus \overline{D}_1)\times [0,1]$ and $f(x,0) = 0$ for any $x \in D_0\setminus \overline{D}_1$. 
	\end{enumerate}

Our main result concerning problem \eqref{main-prob} is the following theorem.

\begin{thm}\label{main-thm1}
	Let $D_0, D_1\subset \R^N$ satisfy (D) and $f: (D_0\setminus \overline{D}_1)\times [0,1] \to \R$. We have:
\begin{enumerate}
\item[(i)] Assume $D_0$ and $D_1$ are bounded starshaped sets and $f$ satisfies (F0), (F1), (F2), (F3). If $u$ is a solution of \eqref{main-prob} such that $0 \le u \le 1$ on $D_0\setminus \overline{D}_1$, then the superlevel sets $U(\ell)$ of $u$ are starshaped for $\ell\in (0,1)$.

\item[(ii)] Assume $D_0$ and $D_1$ are bounded, strictly starshaped sets and $f$ satisfies (F0), (F1), (F2). If $u$ is a solution of \eqref{main-prob} such that $0 < u < 1$ on $D_0\setminus \overline{D}_1$, then the superlevel sets $U(\ell)$ of $u$ are strictly starshaped for $\ell\in (0,1)$.

\item[(iii)] Assume $D_0=\rn$, $D_1$ is a bounded starshaped set, and $f$ satisfies (F0), (F1), (F2), (F3). If $u$ is a solution of \eqref{main-prob} such that $0 \le u \le 1$ on $D_0\setminus \overline{D}_1$, then the superlevel sets $U(\ell)$ of $u$ are starshaped for $\ell\in (0,1)$.
\end{enumerate}
\end{thm}

Note that condition (F1) is analogous to condition (21) from \cite{Salani}. Note also that if $D_1=B_r(0)$ for some $r>0$ in Theorem \ref{main-thm1} (iii) and $f$ is independent of $x$, then it is known that $u$ is radial symmetric and decreasing in the radial direction (see \cite[Theorem 1.10 and Corollary 1.11]{SV14}). In particular, the superlevel sets of $u$ are starshaped.

Using Theorem \ref{main-thm1} and N. Abatangelo's result \cite[Theorem 1.5]{A15} we obtain
\begin{thm}
\label{existence-thm}
Let $D_0, D_1\subset \R^N$ be bounded (strictly) starshaped sets satisfying (D), $D_0\setminus \overline{D}_1$ is a $C^{1,1}$ domain and $f: (D_0\setminus \overline{D}_1)\times [0,1] \to \R$ satisfy (F1), (F2), (F3), (F4). 
Then there exists a unique solution $u$ of \eqref{main-prob}. It satisfies $0 < u <1$ on $D_0\setminus \overline{D}_1$ and all superlevel sets $U(\ell)$ of $u$ are (strictly) starshaped for $\ell\in (0,1)$.
\end{thm}

\begin{remark}
\label{F5F6}
If $f\in C^{1}((D_0\setminus \overline{D}_1)\times [0,1])$ satisfies
\begin{enumerate}
	        \item[(F5)] $f\ge 0$ in $(D_0\setminus \overline{D}_1)\times [0,1]$,
		\item[(F6)] $\langle x,\nabla_x f(x,u)\rangle \ge 0$ for every $(x,u) \in (D_0\setminus \overline{D}_1)\times [0,1]$,
	\end{enumerate}
then it satisfies condition (F1).
Indeed,  let $t \ge 1$ and $(x,u) \in (D_0\setminus \overline{D}_1)\times [0,1]$ be such that $tx \in D_0$;
then 
$$
f(tx,u)-f(x,u)=\int_1^t \frac{d}{ds} f(sx,u) \, ds = \int_1^t\langle x,[\nabla_xf](sx,u)\rangle\ \, ds.
$$
Hence (F5) and (F6) imply (F1).
\end{remark}

As a consequence of Theorem \ref{main-thm1} we obtain the following result for harmonic functions with respect to fractional Laplacians.

\begin{cor}\label{cor1}
Let $D_0, D_1\subset \R^N$ be bounded (strictly) starshaped sets satisfying (D) and $f\equiv 0$. Then there exists a unique solution $u$ of \eqref{main-prob}. It satisfies $0 < u < 1$ on $D_0\setminus \overline{D}_1$ and all superlevel sets $U(\ell)$ of $u$ are (strictly) starshaped for $\ell\in (0,1)$.
\end{cor}
When $D_0\setminus \overline{D}_1$ is sufficiently smooth and $D_0$, $D_1$ are uniformly starshaped we can strengthen the assertion of Corollary \ref{cor1}. 
\begin{thm}\label{uniform-thm}
Let $D_0, D_1\subset \R^N$ be open bounded sets, such that $0\in D_1$ and $\overline{D_1}\subset D_0$. Moreover, assume $D_0$ and $D_1$ are uniformly starshaped, $D_0\setminus \overline{D}_1$ is a $C^{1,1}$ domain and $f\equiv 0$. Then all superlevel sets $U(\ell)$ of solutions $u$ of \eqref{main-prob} are uniformly starshaped for $\ell\in (0,1)$.
\end{thm}

As a consequence of Theorem \ref{main-thm1} we obtain more general result for harmonic functions with respect to Schr{\"o}dinger operators based on fractional Laplacians.

\begin{cor}\label{cor2}
Let $D_0, D_1\subset \R^N$ be bounded (strictly) starshaped sets satisfying (D) and $f(x,u) = q(x) u$, $q$ is a bounded nonnegative Borel function on $(D_0\setminus \overline{D}_1)$ such that 
$$
q(tx) \ge q(x) \quad \text{for any $t > 1$ and $x \in (t^{-1} D_0) \setminus \overline{D_1}$.}
$$
Then there exists a unique solution $u$ of \eqref{main-prob}. It satisfies $0 < u < 1$ on $D_0\setminus \overline{D}_1$ and all superlevel sets $U(\ell)$ of $u$ are (strictly) starshaped for $\ell\in (0,1)$.
\end{cor}

As another consequences of Theorem \ref{main-thm1} and Theorem \ref{existence-thm}  we obtain the following result for Allen-Cahn-type nonlinearities.
\begin{cor}\label{cor3}
Let $D_0, D_1\subset \R^N$ be bounded strictly starshaped sets satisfying (D) and $f(x,u) = \beta u - \gamma u^p$, where $\beta \ge 0$, $\gamma\in \R$ and $p \ge 1$. We have
\begin{enumerate}
\item[(i)] Assume $\beta \ge \gamma$. If $u$ is a solution of \eqref{main-prob} such that $0 < u < 1$ on $D_0\setminus \overline{D}_1$, then the superlevel sets $U(\ell)$ of $u$ are strictly starshaped for $\ell\in (0,1)$.
\item[(ii)] Assume $\beta \ge p \gamma$ and $D_0\setminus \overline{D}_1$ is a $C^{1,1}$ domain. Then there exists a unique solution $u$ of \eqref{main-prob}. It satisfies $0 < u < 1$ on $D_0\setminus \overline{D}_1$ and all superlevel sets $U(\ell)$ of $u$ are strictly starshaped for $\ell\in (0,1)$.
\end{enumerate}
\end{cor}

%The following analogous result to Theorem \ref{main-thm1} (i) holds in a special case when $D_0 = \rn$.

%\begin{thm}\label{main-thm2}
%Let $D_0 = \rn$, $D_1\subset \R^N$ be a bounded starshaped set such that $D_0,D_1$ satisfy (D0). Moreover, let $f: \left(\overline{D_1}\right)^c \times [0,1] \to \R$ satisfy (F0), (F1), (F2) and (F3). If $u$ is a solution of \eqref{main-prob} such that $0 \le u \le 1$ on $D_1^c$ and $\lim\limits_{|x|\to\infty}u(x)=0$, then the superlevel sets $U(\ell)$ of $u$ are starshaped for $\ell\in (0,1)$.
%\end{thm}

We note that Theorem \ref{main-thm1} is in fact a special case of the following more general result in which we do not assume $u$ to be constant on $\overline{D_1}$ and on $\R^N \setminus D_0$. To be precise, let $b_0,b_1$ be continuous and bounded functions on $\rn$ and consider the following problem 
\begin{equation}\label{main-prob2}
\begin{aligned}
-(-\Delta)^{\alpha/2}u&=f(x,u) &&\text{ in $D_0\setminus\overline{D}_1$}\\
u&=b_0&& \text{in $\R^N\setminus D_0$,}\\
u&=b_1&& \text{in $\overline{D}_1$,}
\end{aligned}
\end{equation}
We say that a function $u$ is a {\it{solution}} of (\ref{main-prob2}), if $u$ is continuous and bounded on $\rn$, $u = b_0$ in $\rn \setminus D_0$, $u = b_1$ in $\overline{D}_1$, and $-(-\Delta)^{\alpha/2}u = f(x,u)$ as distributions in $D_0 \setminus \overline{D}_1$.

\begin{thm}\label{main-thm3}
	Let $D_0,D_1\subset \R^N$ satisfy (D), $f: (D_0\setminus \overline{D}_1)\times [0,1] \to \R$ and $b_0, b_1$ be continuous and bounded functions on $\rn$such that $b_1\equiv 1$ on $\partial D_1$ and $b_0\equiv 0$ on $\partial D_0$ and $b_0$ and $b_1$ have starshaped superlevel sets. Then the statements (i), (ii), and (iii) of Theorem \ref{main-thm1} hold for solutions $u$ of \eqref{main-prob2}.	
	%
	%We have:
	%
	%(i) Assume $D_0$ and $D_1$ are bounded and starshaped sets and $f$ satisfies (F0), (F1), (F2), (F3). If $u$ is a solution of \eqref{main-prob2} such that $0 \le u \le 1$ on $D_0\setminus \overline{D}_1$, then the superlevel sets $U(\ell)$ of $u$ are starshaped for $\ell\in (0,1)$.
	%
	%(ii) Assume $D_0$ and $D_1$ are bounded and strictly starshaped sets and $f$ satisfies (F0), (F1), (F2). If $u$ is a solution of \eqref{main-prob2} such that $0 < u < 1$ on $D_0\setminus \overline{D}_1$, then the superlevel sets $U(\ell)$ of $u$ are strictly starshaped for $\ell\in (0,1)$.
	%
	%(iii) Assume $D_0=\rn$ and $D_1$ is a bounded starshaped set and $f$ satisfies (F0), (F1), (F2), (F3). If $u$ is a solution of \eqref{main-prob2} such that $0 \le u \le 1$ on $D_0\setminus \overline{D}_1$, then the superlevel sets $U(\ell)$ of $u$ are starshaped for $\ell\in (0,1)$.
\end{thm}

By similar methods we obtain the following result for Green functions corresponding to fractional Laplacians on  convex bounded domains. For basic properties of the Green functions see Preliminaries. 
\begin{thm}\label{Green-thm}
Let $D \subset \rn$ be an open bounded convex set and $G_D(x,y)$ be the Green function for $D$ corresponding to $(-\Delta)^{\alpha/2}$, $\alpha \in (0,2)$. Then for any fixed $y \in D$ the superlevel sets $U(\ell)$ of the function $u(x) = G_D(x,y)$ are starshaped with respect to $y$ for any $\ell \in (0,\infty)$.
\end{thm}

Let us recall that in the limit case $\alpha=2$, i.e. in the case of the usual Laplacian, these are all well-known results, see for instance \cite{Acker, DK, F, FG, GR, Kawohl, Kaw1, Long, Salani}. Notice that, although the geometric ideas underlying the situation here at hand are similar to the ones of the papers just quoted, we need big efforts to deal with the distinctive peculiarities of the fractional Laplacian. We use completely different methods than in the classical case. Namely, in the proof of our main result Theorem \ref{main-thm3} we study the function $u_t(x) = u(x) - u(tx)$ using the appropriate maximum principle for Schr{\"o}dinger fractional operators. In the proof of Theorem \ref{main-thm3} (ii), in order to relax assumption (F3), we use additionally the method of continuity.
\medskip

The paper is organized as follows. In the very next section we introduce some notation and collect some preliminary results. In \S3 we prove theorems \ref{main-thm1}, \ref{existence-thm}, \ref{main-thm3}, \ref{Green-thm} and their corollaries. Finally, in \S4 we treat uniform starshapedness and prove Theorem \ref{uniform-thm}.
\medskip

{\bf Acknowledgements.} The third author has been partially supported by GNAMPA of INdAM and by the FIR 2013 project `Geometrical and Qualitative aspects of PDE''.

\section{Preliminaries}

Let us fix some notation. In the following $N\in \N$ and $\alpha\in(0,2)$. For $U\subset \R^N$, a nonempty measurable set, we denote by $1_U: \R^N \to \R$ the characteristic function, $|U|$ the Lebesgue measure, and $U^c = \rn \setminus U$ the complement of $U$. The notation $D \subset \subset U$ means that $\overline D$ is compact and contained in the interior of $U$. The distance between $D$ and $U$ is given by $\dist(D,U):= \inf\{|x-y|\::\: x \in D,\, y \in U\}$ and if $D= \{x\}$ we simply write $\dist(x,U)$. Note that this notation does {\em not} stand for the usual Hausdorff distance. We write $\delta_{U}(x)=\dist(x,\R^N\setminus U)$ for the distance function. For $x \in \R^N$, $r>0$, $B_r(x)$ is the open ball centered at $x$ with radius $r$. We also denote $\langle f,g\rangle = \int_{\rn} f g$.

Let $D \subset \rn$ be an open bounded set. By $G_D(x,y)$ we denote the Green function of $D$ with respect to $(-\Delta)^{\alpha/2}$. For any $x \in D$ by $\omega_D^x(dy)$ we denote the harmonic measure  of $D$ with respect to $(-\Delta)^{\alpha/2}$. The definition and basic properties of $G_D(x,y)$ and $\omega_D^x(dy)$ may be found e.g. in \cite[pages 14-15]{book2009}. It is well-known (see e.g. \cite[page 297]{BB2000}) that 
\begin{equation}
\label{Green}
G_D(x,y) = K_{\alpha}(x-y) - \int_{D^c} K_{\alpha}(z-y) \omega_D^x(dz),
\end{equation}
for any $x, y \in D$. 
For $N > \alpha$ the kernel $K_{\alpha}$ denotes the Riesz kernel given by 
$$
K_\alpha(x) = C_{N,\alpha} |x|^{\alpha - N},
$$
where $C_{N,\alpha} = \Gamma((N-\alpha)/2)/(2^{\alpha} \pi^{N/2} \Gamma(\alpha/2))$. For $1 = N \le \alpha$ the kernel $K_{\alpha}$ denotes the so-called {\em compensated Riesz kernel}, given by (see e.g. \cite[page 296]{BB2000})
$$
K_\alpha(x) = \frac{|x|^{\alpha-1}}{2 \Gamma(\alpha) \cos(\pi \alpha/2)}, \quad \text{when $1 = N < \alpha$}
$$
and
$$
K_\alpha(x) = \frac{1}{\pi} \log\frac{1}{|x|}, \quad \text{when $1 = N = \alpha$}.
$$

It is well-known that for any open set $D\subset \R^N$ and $u\in C^2(D)\cap \calL$ the expression $(-\Delta)^{\alpha/2}u(x)$ is well-defined for $x \in D$ and we have the following (see \cite[page 9]{book2009}).
 
\begin{lemma}\label{scaling}
	Let $D\subset \R^N$ be an open bounded set and $u\in C^2(D)\cap \calL$. Then
	\[
	(-\Delta)^{\alpha/2}u(tx)=t^{\alpha} \big[(-\Delta)^{\alpha/2} u\big](tx) \qquad\text{ for all $t>0$ and $x\in t^{-1}D$.}
	\]
\end{lemma}

Let $D \subset \rn$ be an open set, $u \in \calL$ and assume that there exists a bounded Borel function $g: D \to \R$ such that
$$
(-\Delta)^{\alpha/2} u = g
$$
as distributions in $D$. Then for any $t > 1$ we define the distribution $[(-\Delta)^{\alpha/2} u](t\cdot)$ in $t^{-1} D$ by 
$$
[(-\Delta)^{\alpha/2} u](t\cdot) = g(t\cdot).
$$
The following generalization of Lemma \ref{scaling} holds.

\begin{lemma}\label{weakscaling}
Let $D \subset \rn$ be an open set, $u \in \calL$ and assume that there exists a locally integrable Borel function $g: \R^N \to \R$ such that $(-\Delta)^{\alpha/2} u = g$ as distributions in $D$. Let $t > 1$ and put $v(x) = u(tx)$. Then 
	\begin{equation}
	\label{weak}
	t^{\alpha} [(-\Delta)^{\alpha/2} u](t\cdot) = (-\Delta)^{\alpha/2} v,
	\end{equation}
	as distributions in $t^{-1} D$.
\end{lemma}
\begin{remark}
	\label{remarkweakscaling}
	Equivalently (\ref{weak}) may be formulated as
	$$
	t^{\alpha} [(-\Delta)^{\alpha/2}u](tx) = (-\Delta)^{\alpha/2}u(tx) \quad \quad \text{for almost all $x \in t^{-1}D$.}
	$$
\end{remark}
\begin{proof}
Let $\varphi \in C_c^{\infty}(t^{-1}D)$. We have
	$$
	t^{\alpha} \langle[(-\Delta)^{\alpha/2}u](t\cdot), \varphi\rangle\ =\ t^{\alpha} \langle g(t\cdot), \varphi\rangle\ =\ t^{\alpha} \int_{\rn} g(tx) \varphi(x) \, dx.
	$$
	By substitution $y = tx$ this is equal to
	\begin{eqnarray*}
		t^{\alpha - N} \int_{\rn} g(y) \varphi\left(\frac{y}{t}\right) \, dx
		&=& t^{\alpha - N} \left\langle(-\Delta)^{\alpha/2}u,\varphi\left(\frac{\cdot}{t}\right)\right\rangle\\
		&=& t^{\alpha - N} \left\langle u,(-\Delta)^{\alpha/2}\varphi\left(\frac{\cdot}{t}\right)\right\rangle\\
		&=& t^{\alpha - N} \int_{\rn} u(y) (-\Delta)^{\alpha/2}\varphi\left(\frac{y}{t}\right) \, dy.
	\end{eqnarray*}
	By Lemma \ref{scaling} it equals
	$$
	t^{- N} \int_{\rn} u(y) \left[(-\Delta)^{\alpha/2}\varphi\right]\left(\frac{y}{t}\right) \, dy.
	$$
	Substituting $x = y/t$ this finally gives
	$$
	\int_{\rn} v(x) (-\Delta)^{\alpha/2}\varphi(x) \, dx\ =\ \langle v,(-\Delta)^{\alpha/2} \varphi\rangle\ =\ \langle (-\Delta)^{\alpha/2} v, \varphi\rangle.
	$$
\end{proof}
Assume $D \subset \rn$ is an open set, $g, h \in \calL$, $q$, $g$ are bounded Borel functions on $D$ and let us consider the following problem
\begin{equation}
\label{linear-prob}
(-\Delta)^{\alpha/2}u - qu = g \ge 0, \quad \text{on} \, \, D \quad \quad
u = h \ge 0, \quad \quad \text{on} \, \, D^c.
\end{equation}
We say that $u$ is a {\it{solution}} of (\ref{linear-prob}) if $u \in \calL \cap L^{\infty}(D)$, $u$ is continuous on $\overline{D}$,  $(-\Delta)^{\alpha/2}u - qu = g$ as distributions in $D$ and $u=h$ holds pointwise on $D^c$.
To prove our main statements, we use  the following variant of the maximum principle. 

\begin{lemma}\label{wmp}
Let $D\subset \R^N$ be an open bounded set which satisfies a uniform exterior cone condition. Assume that  $q \le 0$ on $D$ and there is an open bounded set $D_0\subset \subset \R^N\setminus D$ such that $h$ is continuous and bounded on $D^c \setminus D_0$. Then a solution $u$ of (\ref{linear-prob}) exists and it is unique. We have $u \ge 0$ on D. If additionally the set $\{x \in D^c: \, h(x) > 0)\}$ has positive Lebesgue measure then $u > 0$ on D.
\end{lemma}
\begin{proof}
The proof will be done in the framework presented in \cite{BB1999,BB2000}. Let $(X_t,P^x)$ denote the standard symmetric $\alpha$-stable process in $\rn$ generated by $-(-\Delta)^{\alpha/2}$. Denote by $E^x$ the expected value corresponding to the process $X_t$ starting from $x$ and let
$$
\tau_D = \inf\{t \ge 0: \, X_t \notin D\}.
$$
be the first exit time from $D$. Moreover, denote $e_q(\tau_D) = \exp\left(\int_0^{\tau_D} q(X_s) \, ds\right),$ and
$$
E^x(e_q(\tau_D)) = E^x\left( \exp\left(\int_0^{\tau_D} q(X_s) \, ds\right)\right) \quad \text{a gauge}.
$$
Since $q \le 0$ on $D$ we get $\sup_{x \in D} E^x(e_q(\tau_D)) \le 1$.
Define
\begin{eqnarray*}
	\tilde{u}(x) &=& E^x(h(X(\tau_D)) e_q(\tau_D)) + Vg(x) , \quad x \in D\\
	\tilde{u}(x) &=& h(x), \quad x \in D^c,
\end{eqnarray*}
where $V$ is the $q$-Green operator corresponding to the Schr{\"o}dinger operator based on the fractional Laplacian (for the formal definition of $V$ see \cite[page 58]{BB1999}). By the gauge theorem (see \cite[page 59]{BB1999}), properties of $h$, $q$ and standard estimates we get $\tilde{u}(x)<\infty$ for every $x\in D$. Clearly, $\tilde{u}(x) \ge 0$ for every $x \in D$. If the set $\{x \in D^c: \, h(x) > 0\}$ has positive Lebesgue measure then $\tilde{u}(x) \ge E^x(h(X(\tau_D) e_q(\tau_D))) > 0$ for any $x \in D$. By \cite[Theorem 4.1]{BB2000} and \cite[Proposition 3.16]{BB1999} $\tilde{u}$ is a solution of (\ref{linear-prob}).

Assume that $\tilde{\tilde{u}}$ is another solution of (\ref{linear-prob}). Put $v = \tilde{u} - \tilde{\tilde{u}}$. Clearly $v \in L^{\infty}(D)$ and it satisfies 
\begin{eqnarray*}
(-\Delta)^{\alpha/2}v - qv &=& 0, \quad \quad \text{on} \, \, D\\
v &=& 0, \quad \quad \text{on} \, \, D^c.
\end{eqnarray*}
By \cite[Remark 6.3]{BB2000} and arguments from the proof of \cite[Lemma 5.4]{BB1999} $v \equiv 0$ on $\rn$.
\end{proof}

\begin{cor}\label{wmp2}
Let $D\subset \R^N$ be an open set which satisfies a uniform exterior cone condition, $q\leq 0$ on $D$ and $h$ is continuous and bounded on 
$D^c$. If $u$ is a solution of \eqref{linear-prob}, and $\liminf\limits_{|x|\to\infty} u(x)\geq0$, then $u\geq 0$ on $D$.
\end{cor}
\begin{proof}
Set $u_{n}:=u+\frac1n$ on $\R^N$ for $n\in\N$. Note that $u_n$ is a solution of \eqref{linear-prob} with $h$, $g$ replaced by $h_n$, $g_n$, where $g_{n} = g - \frac{q}{n}$, $h_n= h + \frac1n$. Since $\liminf\limits_{|x|\to\infty} u(x)\geq0$, for every $n\in\N$ there is $r_n>0$ such that $u_n\geq 0$ on $(B_{r_n}(0)\cap D)^c$ . Set $R_n=r_n+n$. Since $B_{R_n}(0)\cap D$ satisfies again a uniform exterior cone condition, Lemma \ref{wmp} implies $u_n\geq 0$ on $B_{R_n}(0) \cap D$. Hence the claim follows for $n\to +\infty$.
\end{proof}

\begin{remark}
	Note that the above framework is in the sense of distributions. For maximum principles in the variational sense see e.g. \cite{JW14,J16}.
\end{remark}

\section{Starshapedness}

For the sake of completeness we give here the following trivial lemma.
\begin{lemma}\label{lmstar}
Let  $u:\R^N\to\R$ such that $M=\max_{\R^N}u=u(0)$. Then the superlevel sets $U(\ell)$, $\ell\in\R$, of $u$ are all starshaped if and only if $u(tx)\leq u(x)$ for every $x\in\rn$ and every $t\geq 1$.
\end{lemma}
\begin{proof}
Assume $U(\ell)=\{x\in\R^N\,:\,u(x)\geq\ell\}$ is starshaped for every $\ell\in\R$. By \eqref{stardef} this means that $sU(\ell)\subseteq U(\ell)$ for every $s \in [0,1)$.
Now set $tx=y$ and $\ell=u(y)$; then $x=sy$ where $s=t^{-1}\in(0,1]$, whence $x\in U(\ell)$, i.e. $u(x)\geq\ell=u(tx)$.

Conversely, assume $u(x)\geq u(tx)$ for every $x\in\R^N$ and every $t\geq 1$. Now take $\ell\in\R$: if $\ell\leq\inf_{\R^N}u$ or $\ell>M$ there is nothing to prove. Then let $\inf_{\R^N}u<\ell\leq M$. The superlevel set 
$U(\ell)=\{x\in\R^N\,:\,u(x)\geq\ell\}$ is starshaped if and only if $sU(\ell)\subseteq U(\ell)$ for every $s\in[0,1]$, see \eqref{stardef}.
If $s=0$ it is trivial, otherwise let $x\in U(\ell)$, that is $u(x)\geq\ell$: we want to prove that $y=sx\in U(\ell)$ as well, i.e. $u(y)\geq\ell$. But 
$x=ty$ where $t=s^{-1}\geq 1$, then $\ell\leq u(x)=u(ty)\leq u(y)$ and the prove is complete.
\end{proof}
\medskip
Now we can proceed to the proofs of our main results.

\begin{proof}[Proof of Theorem \ref{main-thm3}]
For any $t>1$ set 
$$u_t(x)=u(x)-u(tx)\quad x\in\R^N\,.$$
Thanks to Lemma \ref{lmstar}, the starshapedness of the level sets of $u$  is equivalent to  
\begin{equation}\label{ut<0}
u_t\geq 0\,\,\text{ in }\R^N\,\,\text{ for }\,\,t>1\,.
\end{equation}

Observe that since the superlevel sets of $b_0$ and $b_1$ are starshaped, we have $u_t\geq 0$ in $\R^N\setminus D_0$ and in $t^{-1}\overline{D}_1$ and
\begin{equation}\label{utboundary}
u_t(x)\geq 0\quad\text{for }x\in D_0\setminus (t^{-1} D_0)\,\text{ and }\, x\in \overline{D}_1\setminus(t^{-1}\overline{D}_1)\,.
\end{equation}

Put $D_t=(t^{-1}D_0)\setminus\overline{D}_1$. It remains to investigate $u_t$ in $D_t$. Note that if $D_0$ is bounded then for $t$ large enough $D_t$ is empty.

{\bf{Proof of (i).}}
By Lemma \ref{weakscaling} and Remark \ref{remarkweakscaling} we get
\begin{eqnarray*}
(-\Delta)^{\alpha/2}u_t(x)&=&(-\Delta)^{\alpha/2}u(x)-t^{\alpha} \big[(-\Delta)^{\alpha/2} u\big](tx)\\
&=&t^\alpha f(tx,u(tx)) -f(x,u(x))\\
&=&t^\alpha f(tx,u(tx)) - f(x,u(tx)) + f(x,u(tx)) -f(x,u(x)).
\end{eqnarray*}
for almost all $x \in D_t$. For $x \in D_t$ put
$$
q_t(x) = \left\{\begin{array}{cc}      
\dfrac{f(x,u(tx)) -f(x,u(x))}{u_t(x)}, & \text{when} \quad u_t(x) \ne 0, \\ 
0,  & \text{when} \quad u_t(x) = 0.  
 \end{array}\right.
$$
Clearly, $f(x,u(tx)) -f(x,u(x)) = q_t(x) u_t(x)$.

Thus by (F1) we have
$$
(-\Delta)^{\alpha/2}u_t(x) - q_t(x)\,u_t(x) = t^\alpha f(tx,u(tx)) - f(x,u(tx)) \ge 0.
$$
By (F2) $|q_t(x)| \le C$ for $x \in D_t$. By (F3) $q_t(x) \le 0$ for $x \in D_t$.
Recall that $u_t(x) \ge 0$ for $x \in D_t^c$.
Lemma \ref{wmp} implies $u_t(x) \ge 0$ for $x \in D_t$. This finishes the proof in case (i).

{\bf{Proof of (ii).}}
In this case it is enough to show that 
\begin{equation}
\label{utpositive}
u_t(x) > 0 \quad \text{for all $t > 1$  and $x \in D_t$.}
\end{equation}
Put 
$$
t_0 = \sup\{s \in (1,\infty): \, D_s \,\, \text{is not empty}\},
$$
$$
A = \{s \in (1,\infty): \, \text{there exists} \,\, x \in D_s \,\, \text{such that} \,\, u_s(x) \le 0\}
$$
and
$$
t = \sup A.
$$
We put $t = -\infty$ if the set $A$ is empty. 
By strict starshapedness of $D_0$, $D_1$, the fact that $0 < u < 1$ on $D_0 \setminus \overline{D_1}$ and continuity of $u$ we get that $t < t_0<\infty$ (since $D_0$ is bounded).

On the contrary, assume that (\ref{utpositive}) does not hold. Then the set $A$ is not empty so $t > 1$.
Using strict starshapedness of $D_0$, $D_1$ we obtain $u_t(x) > 0$ for $x \in \partial D_t$. By continuity, $u_t(x) \ge 0$ for $x \in D_t$ and there exists $x_0 \in D_t$ such that $u_t(x_0) = 0$.
 
Similarly as before, by Lemma \ref{weakscaling} and Remark \ref{remarkweakscaling} we get
\begin{eqnarray*}
(-\Delta)^{\alpha/2}u_t(x)&=&(-\Delta)^{\alpha/2}u(x)-t^{\alpha} \big[(-\Delta)^{\alpha/2} u\big](tx)\\
&=&t^\alpha f(tx,u(tx)) - f(x,u(tx)) + f(x,u(tx)) -f(x,u(x)).
\end{eqnarray*}
for almost all $x \in D_t$. For $x \in D_t$ put $F(x,t) = f(x,u(tx)) -f(x,u(x))$, $F_+(x,t) = \max(0,F(x,t))$, $F_-(x,t) = \max(0,-F(x,t))$ and 
$$
q_t(x) = \left\{\begin{array}{cc}      
\dfrac{-F_-(x,t)}{u_t(x)}, & \text{when} \quad u_t(x) \ne 0, \\ 
0,  & \text{when} \quad u_t(x) = 0.  
 \end{array}\right.
$$
We have $f(x,u(tx)) -f(x,u(x)) = F_+(x,t) + q_t(x) u_t(x)$.

Using (F1) we obtain
$$
(-\Delta)^{\alpha/2}u_t(x) - q_t(x)\,u_t(x) = t^\alpha f(tx,u(tx)) - f(x,u(tx)) + F_+(x,t) \ge 0.
$$
Clearly, $q_t(x) \le 0$ for $x \in D_t$. By (F2) $|q_t(x)| \le C$ for $x \in D_t$. Note that $u_t(x) > 0$ for $x \in D_1 \setminus t^{-1}D_1$. Clearly, $D_1 \setminus t^{-1}D_1$ has positive Lebesgue measure. Recall that $u_t(x) \ge 0$ for $x \in D_t^c$.
Lemma \ref{wmp} implies $u_t(x) > 0$ for $x \in D_t$. This contradicts $u_t(x_0) = 0$. So (\ref{utpositive}) holds.

{\bf{Proof of (iii).}}
	The proof proceeds exactly as the one of (i), but we use Corollary \ref{wmp2} in place of Lemma \ref{wmp}.
\end{proof}

\begin{proof}[Proof of Theorem \ref{main-thm1}]
Choosing $b_0\equiv 0$ and $b_1\equiv 1$ in Theorem \ref{main-thm3} gives Theorem \ref{main-thm1}.
\end{proof}

\begin{proof}[Proof of Theorem \ref{existence-thm}]
We use \cite[Theorem 1.5]{A15}. We extend $f$ by putting $f(x,u) = f(x,0)=0$ for $u < 0$ and $f(x,u) = f(x,1)$ for $u > 1$ ($x \in D_0\setminus \overline{D_1}$). As a subsolution we take $\underline{u} = 1_{\overline{D_1}}$, as a supersolution we take $\overline{u} = 1_{D_0}$. By \cite[Theorem 1.5]{A15} there exists a unique weak solution $u$ of (\ref{main-prob}) in the sense of \cite[Definition 1.3]{A15}. This solution satisfies $0 \le u \le 1$. Put $D = D_0\setminus \overline{D_1}$. By \cite[Theorem 1.4]{A15} we have 
\begin{equation}
\label{green_rep}
u(x) = -\int_D G_D(x,y) f(y,u(y)) \, dy + h(x), \quad x \in \rn, 
\end{equation}
where $h$ is the unique continuous solution of 
\begin{equation*}
	\begin{aligned}
		-(-\Delta)^{\alpha/2}h&=0 &&\text{ in $D$}\\
		h&=0&& \text{in $\R^N\setminus D_0$,}\\
		h&=1&& \text{in $\overline{D}_1$.}
	\end{aligned}
\end{equation*}
By \cite[Lemma 5.3]{BB2000} we get (\ref{weak1}). It is well known \cite[page 57]{BB1999} that $\int_D G_D(x,y) f(y,u(y)) \, dy$ is continuous on $D$. Hence $u$ is a solution of (\ref{main-prob}). Now we show that $0<u<1$ in $D$. Since $f(x,0)=0$ for all $x\in D$ by (F4), $u\geq0$ in $D$ we have by (F2) that $q:D\to \R$, 
$$q(x)=\left\{\begin{array}{rl}-\frac{f(x,u(x))}{u(x)}\,\,&\text{ for }x\in D\text{ such that }u(x)\neq0\\
0\,\,&\text{otherwise}\,,\end{array}\right.
$$ 
is a bounded function which satisfies due to (F3) $q\leq 0$ in $D$. Hence for a.e. $x\in D$ we have
\[
(-\Delta)^{\alpha/2}u(x)-q(x)u(x)=-f(x,u(x))-q(x)u(x)=0.
\]
Since $u\equiv 1$ in $\overline{D_1}$, Lemma \ref{wmp} implies $u>0$ in $D$. Moreover, for $v=1-u$ we have for a.e. $x\in D$
\[
(-\Delta)^{\alpha/2}v(x)=-(-\Delta)^{\alpha/2}u(x)=f(x,u(x))\geq0
\]
Since $v\equiv 1$ in $\R^N\setminus D_0$ and $v\equiv 0$ in $D_1$, Lemma \ref{wmp} implies $v>0$ in $D$ and thus $1>u$ in $D$ as claimed. The assertions on the shape of the superlevel sets of $u$ now follow from Theorem \ref{main-thm1}.
\end{proof}

\begin{proof}[Proof of Corollary \ref{cor1}]
It is well known that there exists a unique solution of \eqref{main-prob} (see e.g. \cite[(1.49), (1.53)]{book2009}, \cite[page 57]{BB1999}). The assertion follows from Theorem \ref{main-thm1}.
\end{proof}

\begin{proof}[Proof of Corollary \ref{cor2}]
Under the assumptions of this corollary, by the arguments from \cite{BB1999}, it is well-known that there exists a unique solution $u$ of \eqref{main-prob} which satisfies $0 < u < 1$ in $D_0\setminus \overline{D}_1$. The conditions (F0), (F1), (F2) and (F3) are clearly satisfied so the assertion follows from Theorem \ref{main-thm1}.
\end{proof}

%\begin{proof}[Proof of Theorem \ref{main-thm2}]
%The proof proceeds exactly as the one of Theorem \ref{main-thm1} (i), but we use Corollary \ref{wmp2} in place of Lemma \ref{wmp} .
%\end{proof}

\begin{proof}[Proof of Corollary \ref{cor3}]
If $\beta\geq \gamma$ then clearly $f$ satisfies (F0), (F1), (F2). Hence (i) follows from Theorem \ref{main-thm1}(ii). If $\beta\geq \gamma p$ then $\partial_u f(x,u)= \beta-\gamma pu^{p-1}\geq 0$. Hence (F3) is satisfied and (ii) follows from Theorem \ref{existence-thm}.
\end{proof}

\begin{proof}[Proof of Theorem \ref{Green-thm}]

Case 1. $N > \alpha$.

We may assume that $y = 0$. Clearly $u(x) = 0$ when $x \notin D$.
For any $t>1$ set 
$$u_t(x)=u(x)-u(tx)\quad \text{for} \quad x\in\R^N \setminus \{0\} \quad \text{and} \quad u_t(0) = 0.$$
Then, thanks to Lemma \ref{lmstar}, the statement is equivalent to prove that 
$u_t\geq 0$ in $\rn$ for $t>1$.

Observe that $u_t\equiv 0$ in $D^c$ and $u_t(x) \ge 0$ for $x \in D \setminus (t^{-1} D)$. Put $D_t=(t^{-1}D)\setminus\{0\}$. 

Fix $t > 1$.
Put 
$$
h_D(x) = \int_{D^c} K_{\alpha}(z) \omega_D^x(dz), \quad x \in D.
$$
It is clear that 
$$
h_{D}(x) \le C_{N,\alpha} (\dist(0,D^c))^{\alpha-N}, \quad x \in D.
$$
By (\ref{Green}) for any $x \in D_t$ we have
$$
u(x) = \frac{C_{N,\alpha}}{|x|^{N - \alpha}} - h_D(x)
$$
and
$$
u_t(x) = \frac{C_{N,\alpha}}{|x|^{N - \alpha}} \left(1 - \frac{1}{t^{N - \alpha}}\right) - h_D(x) + h_D(tx).
$$
It is obvious that there exists $\varepsilon = \varepsilon(N,\alpha,\dist(0,D^c),t) > 0$ such that $\overline{B_{\varepsilon}(0)} \subset D_t$ and 
$$
u_t(x) > 0 \quad \text{for} \quad x \in B_{\varepsilon}(0) \setminus\{0\}.
$$

By Lemma \ref{scaling} we get
$$
(-\Delta)^{\alpha/2}u_t(x)=(-\Delta)^{\alpha/2}u(x)-t^{\alpha} \big[(-\Delta)^{\alpha/2} u\big](tx)= 0, 
$$
for $x \in D_t \setminus \overline{B_{\varepsilon}(0)}$.
Since $u_t$ is bounded on $D_t \setminus \overline{B_{\varepsilon}(0)}$, we can apply Lemma \ref{wmp}
to get $u_t\geq 0$ in $\rn$.

Case 2. $1 = N \le \alpha$.

The only bounded convex sets in $\R$ are bounded intervals. By scaling we may assume that $D=(-1,1)$. It is well known (see \cite{BGR1961}) that
\begin{equation}
\label{Greenf}
G_D(x,y) = \cBa |x - y|^{\alpha - 1} \int_0^{w(x,y)} \frac{r^{\alpha/2-1}}{(r+1)^{1/2}} \, dr, \quad \quad x,y \in D, \,\, x \ne y,
\end{equation}
where
$$
w(x,y) = (1-x^2)(1-y^2)/(x-y)^2,
$$
and $\cBa = 1/(2^{\alpha}\Gamma^2(\alpha/2))$. If $1 = N = \alpha$ we have $G_D(x,x) = \infty$, $x \in D$. If $1= N < \alpha$ then $G_D(x,y)$ is bounded and continuous on $D \times D$ and for $x \in D$ we have $G_D(x,x) = (1 - x^2)^{\alpha - 1}/(2^{\alpha-1}\Gamma^2(\alpha/2)(\alpha-1))$ \cite[page 298]{BB2000}.

For $1 = N = \alpha$ the assertion follows by direct computation. Indeed, for $x, y \in D$, $x \ne y$ we have
$$
\frac{\partial}{\partial x} G_D(x,y) = \cBa \frac{(2 x y - 2)(1-y^2)}{(x-y)^3} \frac{w(x,y)^{\alpha/2-1}}{(w(x,y)+1)^{1/2}}.
$$
So the function $x \to G_D(x,y)$ is increasing on $(-1,y)$ and decreasing on (y,1).

Assume now $1= N < \alpha$. Substituting $t = r(x-y)^2$ in (\ref{Greenf}) we obtain
$$
G_D(x,y) = \cBa \int_0^{(1-x^2)(1-y^2)} \frac{t^{\alpha/2-1}}{(t+(x-y)^2)^{1/2}} \, dt, \quad \quad x,y \in D, \,\, x \ne y.
$$
Hence for $x, y \in D$, $x \ne y$ we have
\begin{eqnarray*}
\frac{\partial}{\partial x} G_D(x,y) &=& \cBa \frac{-2x (1 - y^2) ((1-x^2)(1-y^2))^{\alpha/2-1}}{((1-x^2)(1-y^2)+(x-y)^2)^{1/2}} \\
&& + \cBa (y - x) \int_0^{(1-x^2)(1-y^2)} \frac{t^{\alpha/2-1}}{(t+(x-y)^2)^{3/2}} \, dt.
\end{eqnarray*}
So for $x \in D \setminus \{y\}$ such that $|y - x|$ is sufficiently small $\frac{\partial}{\partial x} G_D(x,y)$ behaves like $(y - x) | y - x|^{\alpha - 3}$. In particular, for $x \in D \setminus \{y\}$ such that $|y - x|$ is sufficiently small the function $x \to G_D(x,y)$ is increasing for $x < y$ and decreasing for $x > y$. The rest of the proof is similar to the proof in case $N > \alpha$ and it is omitted.
\end{proof}

\section{Uniform starshapedness}

\begin{lemma}\label{computation}
	Let $D\subset \R^N$ open and $u\in C^{3}(D)\cap \cL^1_{\alpha}\cap W^{1,1}_{loc}(\R^N)$ such that $x\mapsto \langle x, \nabla u(x)\rangle\in C^2(D)$ and
$$
\int_{\rn} \frac{|\langle x, \nabla u(x)\rangle|}{(1 + |x|)^{N + \alpha}} \, dx < \infty.
$$
Then 
	$$
	(-\Delta)^{\alpha/2}\langle x, \nabla u\rangle=\alpha (-\Delta)^{\alpha/2}u+ \langle x, \nabla (-\Delta)^{\alpha/2} u \rangle \quad\text{ on $D$.}
	$$
\end{lemma}
\begin{proof}
	Note that if $\nabla u(x)$ exists then $\partial_t u(tx)|_{t=1}=\langle x, \nabla u(x)\rangle$. By Lemma \ref{scaling} we have for $x\in D$
	\begin{align*}
	\alpha (-\Delta)^{\alpha/2}u(x)+ \langle x, \nabla (-\Delta)^{\alpha/2} u(x)\rangle&=\partial_t \big(t^{\alpha}[(-\Delta)^{\alpha/2}u](tx)\big) |_{t=1}\\
	&=\partial_t \big((-\Delta)^{\alpha/2} u(tx)\big)|_{t=1}.
	\end{align*}
	Hence, it is enough to show that the function $v(t,x)=u(tx)$, $t>0$, $x\in \R^N$ satisfies
	\begin{equation}\label{interchange-claim}
	\partial_t  (-\Delta)^{\alpha/2} v(t,x) = (-\Delta)^{\alpha/2}\partial_tv(t,x) \quad \text{ for $ t>0$,  $x\in t^{-1}D$.}
	\end{equation}
	The argument will be similar to \cite[Proposition B.2]{AJS16}. By the regularity of $u$, we have $v(t,\cdot)\in C^3(t^{-1}D)\cap \cL^1_{\alpha}\cap W^{1,1}_{loc}(\R^N)$, $\partial_t v(t,\cdot)\in C^2(t^{-1}D)\cap \cL^1_{\alpha}$ for every $t>0$, and thus for $t>0$, $x\in t^{-1}D$
	\[
	(-\Delta)^{\alpha/2} v(t,x)=\frac{c_{N,\alpha}}{2}\int_{\R^N} \frac{2v(t,x)-v(t,x+y)-v(t,x-y)}{|y|^{N+\alpha}}\ dy.
	\]
	Define $a,a_h:\{(t,x)\;:\; t>0,\ x\in t^{-1}D\}\times \R^N\setminus \{0\}\to \R$ as
	\[
	a(t,x,y)= \frac{2v(t,x)-v(t,x+y)-v(t,x-y)}{|y|^{N+\alpha}},\quad a_h(t,x,y)=\frac{a(t+h,x,y)-a(t,x,y)}{h},\quad h\in \R\setminus\{0\}
	\]
	and fix $t>0$, $x\in t^{-1}D$. Since $D$ is open, we may fix $U=B_H(0)$, $H>0$, $H < t/2$ such that $(t+h)(x+y)\in D$ for all $y\in U$ and $h\in(-H,H)$. We will show separately
	\begin{align}
	\lim_{h\to 0} \int_{U}a_h(t,x,y)\ dy&= \int_{U}\partial_t a(t,x,y)\ dy\quad \text{ and } \label{interchange:part1}\\
	\lim_{h\to0}  \int_{U^c}a_h(t,x,y)\ dy&= \int_{U^c}\partial_t a(t,x,y)\ dy. \label{interchange:part2}
	\end{align}
	By the Mean Value Theorem, for every $0<|h|<H$ and $y \in B_H(0)$ there is $\xi\in (-|h|,|h|)$ such that $a_h(t,x,y)=\partial_t a(t + \xi,x,y)$. Hence
	\[
	|a_h(t,x,y)|\leq c(t) \|u\|_{C^3(U)}|y|^{2-\alpha-N}\in L^1(U) \quad \text{ for all $0<|h|<H$,}
	\]
	where we used the $C^2$ estimate of $\partial_t v(t,\cdot)$ as given in \cite[Lemma B.1]{AJS16}. Hence \eqref{interchange:part1} holds by the Dominated Convergence Theorem. To see \eqref{interchange:part2}, denote $A:=\{y\in \R^N\;:\; |x-y|\geq \epsilon\}$ for $\epsilon>0$ and note that there is $K>0$, depending on $u$, $\epsilon$, $N$, $\alpha$, $x$, and $U$, so that
	\[
	\Bigg|\frac{\partial_t v(t,y)}{|x-y|^{N+\alpha}}1_A(y)\Bigg| = \frac{|\partial_t v(t,y)|}{1+|y|^{N+\alpha}} \frac{1+|y|^{N+\alpha}}{|x-y|^{N+\alpha}}1_A(y)\leq K\frac{|\partial_t v(t,y)|}{1+|y|^{N+\alpha}} 1_A(y)=:f(y).
	\]
	Indeed, the existence of $K$ is clear for $|x-y|\geq \frac{|y|}{2}$ and if $|x-y|< \frac{|y|}{2}$, then $\frac{|y|}{2}\leq |x|$ and hence $K$ can be chosen depending on $x$. Hence, since $f\in L^1(\R^N)$, we have by the Dominated Convergence Theorem
	\begin{align*}
	\lim_{h\to0} \int_{U^c} \frac{v(t+h,x\pm y) -v(t,x\pm y)}{h|y|^{N+\alpha}}\ dy&=  \int_{U^c}\lim_{h\to0} \frac{v(t+h,x\pm y) -v(t,x\pm y)}{h} \frac{1}{|y|^{N+\alpha}}\ dy\\
	&=\int_{U^c} \frac{\partial_tv(t,x\pm y)}{|y|^{N+\alpha}}\ dy
	\end{align*}
	using the fact that $u\in W^{1,1}_{loc}(\R^N)$.	Moreover, trivially
	\[
	\lim_{h\to0} \int_{U^c} \frac{v(t+h,x) -v(t,x)}{h|y|^{N+\alpha}}\ dy= \int_{U^c} \frac{\partial_tv(t,x)}{|y|^{N+\alpha}}\ dy
	\]
	and thus \eqref{interchange:part2} holds. Finally, \eqref{interchange:part1}, \eqref{interchange:part2} immediately imply \eqref{interchange-claim} and this finishes the proof.
\end{proof}

\begin{proof}[Proof of Theorem \ref{uniform-thm}]
Put $D = D_0 \setminus \overline{D_1}$ and $w(x) = \langle x, \nabla u(x)\rangle$. Recall that $\delta_D(x) =\dist(x,D^c)$. By $\nu(x)$ we denote the exterior unit normal for $D_0$ at $x \in \partial D_0$ and the exterior unit normal for $D_1$ at $x \in \partial D_1$. By the uniform starshapedness assumption, for any $x \in \partial D$ we have $\langle x, \nu(x)\rangle \ge \varepsilon$ for some $\varepsilon > 0$. Note that it remains to show that $|\nabla u(x)| > 0$ and there exists $c > 0$ such that 
\begin{equation}
\label{nabla_ineq}
\frac{w(x)}{|\nabla u(x)|} \le -c \quad \text{for any $x \in D$.}
\end{equation}
By well known properties of $C^{1,1}$ domains there exists $r_1 > 0$ such that for all $x \in D$ with $\delta_D(x) \le r_1$ there exists a unique point $x^* \in \partial D$ such that $|x-x^*| = \delta_D(x)$. For any $x \in D$ such that $\delta_D(x) \le r_1$ put $\nu(x) = \nu(x^*)$. There exists $r_2 \in (0,r_1]$ such that for all $x \in D$ with $\delta_D(x) \le r_2$ we have $\langle x, \nu(x)\rangle \ge \varepsilon/2$. By \cite[Lemma 4.5]{BKN2002} and standard arguments as in \cite[proof of Lemma 3.2]{K2014} there exists $r_3 \in (0,r_2]$ and $c > 0$ such that for all $x \in D$ with $\delta_D(x) \le r_3$ we have $\langle \nu(x), \nabla u(x)\rangle \le - c \delta_D^{\alpha/2 -1}(x)$, and $|\langle t(x), \nabla u(x)\rangle| \le c \max(\delta_D^{1-\alpha/2}(x),\delta_D^{\alpha/2}(x))  |\log(\delta_D(x))|$, for any vector $t(x)$ perpendicular to $\nu(x)$. Hence there there exists $r_4 \in (0,r_3]$ such that for all $x \in D':=\{x\in D\;:\; \delta_D(x) \le r_4\}$ we have $|\nabla u(x)| > 0$ and 
\begin{equation}
\label{nabla_b}
\frac{w(x)}{|\nabla u(x)|} \le -\varepsilon/4.
\end{equation} 
Now we will use Lemma \ref{computation} for $u$. By similar arguments as in the proof of \cite[Proposition 1.1]{RS12} we get $u \in C^{\alpha/2}(\rn)$. Other assumptions on $u$ in Lemma \ref{computation} are clearly satisfied. By this lemma we obtain $(-\Delta)^{\alpha/2} w(x) = 0$ in $D\setminus D'$. We also have $w \leq 0$ on $\inn( D^c\cup D')$ and $|\{w<0\}\cap D'|>0$ by \eqref{nabla_b}. Since $D_0$ is bounded, Lemma \ref{wmp} implies $w<0$ in $D\setminus D'$. Hence with \eqref{nabla_b} and the continuity of $w$ in $\overline{D\setminus D'}$ there exists some $c_1 > 0$ such that $w(x) \le - c_1$ in $D$ implying (\ref{nabla_ineq}).
\end{proof}

\end{document}